\documentclass[reqno,10pt]{amsart}

\usepackage{amsmath,amssymb,graphicx}

\usepackage[latin1]{inputenc}
\usepackage[english]{babel}
\usepackage{cases}

\newtheorem{Theorem}{Theorem}[section]

\newtheorem{corollary}[Theorem]{Corollary}
\newtheorem{lemma}[Theorem]{Lemma}

\theoremstyle{definition}
\newtheorem{Remark}[Theorem]{Remark}
\newtheorem{definition}[Theorem]{Definition}

\title[Extensions of character formulas]{Extensions of character formulas by the Littlewood decomposition}

\author[Mathias P\'etr\'eolle]{Mathias P\'etr\'eolle}
\address{Institut Camille Jordan, Universit\'e Claude Bernard Lyon 1,
69622 Villeurbanne Cedex, France}
\email{petreolle@math.univ-lyon1.fr}
\urladdr{http://math.univ-lyon1.fr/{\textasciitilde}petreolle}

\begin{document}
\maketitle

\begin{abstract}
 In 2015, the author proved combinatorially character formulas expressing sums of the (formal) dimensions of irreducible representations of symplectic groups, refining some works of Nekrasov and Okounkov, Han, King, and Westbury. In this article, we obtain generalizations of these character formulas, by using a bijection on integer partitions, namely the Littlewood decomposition, for which we prove new properties. As applications, we derive signed generating functions for subsets of integer partitions, and new hook length formulas.

\end{abstract}

\section{Introduction}

Studying irreducible characters of classical Lie groups is classical, as they link many domains of mathematics, such as combinatorics, number theory, or representation theory  (see for instance \cite{KRATT} and the references cited there). An example of such a connection  is implicitly given by results of King \cite{KING89} and King--El-Samra \cite{KS}. Indeed, in \cite{KING89}, a modification rule for the characters of  the Lie algebra of type $\widetilde{A}$ is proved. In \cite{KS} the dimensions of formal irreducible characters of $SU$ are computed. Combining these results, one can equate an arbitrary power of the Dedekind $\eta$ function with a combinatorial sum over integer partitions, involving hook lengths. Here, recall that the Dedekind $\eta$ function is a weight $1/2$ modular form defined as $ \eta(x) = x^{1/24} \prod_{k \geq 1} (1-x^k)$. This gives for instance the now well-known Nekrasov--Okounkov formula \cite[Equation (6.12)]{NO}, that Nekrasov and Okounkov obtained through the Seiberg--Witten theory in 2006 and which was also proved independently by Westbury in \cite{WEST}.

Another proof of this formula, having a  more combinatorial flavour, is due to Han and uses both Macdonald formula for the affine root systems of type $\widetilde{A}$ \cite[Equation (1.3)]{HAN} and a famous bijection due to Garvan--Kim--Stanton \cite{GKS}.

Following Han's approach and using Macdonald formula for affine root systems of types $\widetilde{C}$, the author proved in {\cite[Theorem 1.1]{PET2}}  the following  Nekrasov--Okounkov type formula in type  $\widetilde{C}$. For any complex number $z$, with the notations and definitions of Section~\ref{section2}, we have:
\begin{equation}\label{eqtheoremeintro}
\prod_{k \geq 1}(1-x^k) ^{2z^2+z} = \sum_{\lambda \in DD}\delta_\lambda\,   x^{|\lambda|/2} \prod_{h \in  \mathcal{H}(\lambda)} \left( 1- \frac{2z+2}{h\, \varepsilon_h }\right), 
\end{equation}
where the sum is over doubled distinct  partitions $\lambda$, $\delta_\lambda$ is equal to $1$ (respectively $-1$) if the Durfee square of $\lambda$ is  of even (respectively odd) size, and $\varepsilon_h$ is equal to $-1$ if $h$ is the hook length of a box strictly above the diagonal in the Ferrers diagram of $\lambda$ and to $1$ otherwise.

The formula expressed in the following theorem, is a new Nekrasov--Okounkov type formula, and corresponds to type $\widetilde{C}\check{~}$. Actually, these two results generalize most of Macdonald formulas in the classical affine types $\widetilde{B}, \widetilde{B}\check{~}, \widetilde{C}, \widetilde{C}\check{~}, \widetilde{BC}, \widetilde{D}$.

\begin{Theorem}\label{theoremeintro} 
For any complex number $z$,  the following expansion holds:
\begin{equation}
\label{eqgenccheck2}
\left(\prod_{k \geq 1} \frac{(1-x^{2k})^{z+1}}{1-x^k}\right)^ {2z-1}= \sum _{\lambda \in SC} \delta_\lambda \, x^{|\lambda|} \prod_{h \in \mathcal{H}(\lambda)}\left(1-\frac{2z}{h\, \varepsilon_h}\right), \text{ (type $\widetilde{C}\check{~}$)}
\end{equation}
where the sum ranges self-conjugate partitions $\lambda$, $\delta_\lambda$ is equal to $1$ (respectively $-1$) if the Durfee square of $\lambda$ is  of even (respectively odd) size, and $\varepsilon_h$ is equal to $-1$ if $h$ is the hook length of a box strictly above the diagonal in the Ferrers diagram of $\lambda$ and to $1$ otherwise.

\end{Theorem}

We prove this theorem in a combinatorial way, following Han's approach, by using a Macdonald formula in type $\widetilde{C}\check{~}$ and by introducing a notion of generalized hook length.
Here again, it appears that when $z$ is a positive integer, Theorem~\ref{theoremeintro} can also be proved algebraically, by using results of \cite{KING89} and \cite{KS}. More precisely, Equations \eqref{eqtheoremeintro} and \eqref{eqgenccheck2} become in this case character formulas which express sums of the (formal) dimensions of irreducible representations of symplectic groups.

In another direction, Han generalized the Nekrasov--Okounkov formula by adding two extra parameters (one being a positive integer and the other one a complex number)  through  the Littlewood decomposition, which is a classical bijection between partitions and some vectors of partitions \cite{HAN,HJ}. In the same way, the author generalizes \eqref{eqtheoremeintro} in \cite[Theorem 1.2]{PET2} by adding two extra parameter, however only in the case where $t=2t'+1$ is a odd positive integer and $z$ a complex number:
\begin{multline}
\sum_{\lambda \in DD} \delta_{\lambda}\, x^{|\lambda|/2} \prod_{ h \in \mathcal{H}_t(\lambda)}\left( y -\frac{ytz}{ h\, \varepsilon_h}\right) \\= \displaystyle\prod_{k \geq 1}  (1-x ^k)(1-x^{kt})^ {t'-1} \left(1-y^{2k}x ^{kt}\right)^{(z-1)(zt+t-3)/2}. \label{eq1}
\end{multline}

 It is natural to ask whether one can also generalize Theorem~\ref{theoremeintro} and \eqref{eqtheoremeintro} in the $t$ even case through this bijection. The answer is positive, as we will see in the present paper.

 To do this, we will need new properties of this decomposition. Indeed, as the partitions involved in Theorem~\ref{theoremeintro} and \eqref{eqtheoremeintro} have special forms, we have to study their images under the Littlewood decomposition. Some specificities of the images (according to the parity of the aforementioned positive integer) will lead us to two different generalizations for each formula  \eqref{eqtheoremeintro} and \eqref{eqgenccheck2}. We state below the one corresponding to \eqref{eqtheoremeintro} in the $t$ even case.
\begin{Theorem}\label{generalisation}
Let $t=2t'$ be an even positive integer.  With the notations and definitions of Section~\ref{section2}, we have:
\begin{multline}
\sum_{\lambda \in DD} \delta_{\lambda}\, x^{|\lambda|/2} \prod_{ h \in \mathcal{H}_t(\lambda)}\left( y -\frac{ytz}{ h\, \varepsilon_h}\right)
\\= \displaystyle \prod_{k \geq 1}\frac{(1-x^k)(1-x^{kt})^{t'-1}}{1-x^{kt'}}\left(\frac{(1-y ^{2k}x^{kt})^{zt'-1+t'}}{1-y^kx^{kt'}}\right)^{z-1},\label{eq2}
\end{multline}
where the sum ranges over doubled distinct partitions, and $\mathcal{H}_t(\lambda)$ is the multi-set of hook lengths of $\lambda$ which are integral multiples of $t$.
\end{Theorem}

We now give the two generalizations of Theorem~\ref{theoremeintro}.

\begin{Theorem}\label{genscpair}
Let $t$  be a positive integer. For any complex numbers $y$ and $z$, we have:
\begin{multline*}
\sum_{\lambda \in SC} \delta_\lambda \, x^ {|\lambda|} \prod_{h \in \mathcal{H}_t(\lambda)}\left(y - \frac{yzt}{h \, \varepsilon_h}\right)
\end{multline*}

\begin{numcases}{=} \displaystyle \prod_{k \geq 1} \frac{\left({1-x^{2kt}}\right)^{t'}}{1+x^k}(1-y^{2k}x^{2kt})^{(z^2-1)t'}, \text{\hspace*{68pt} if $t=2t'$,} \label{eqgenscpair}
\\ \displaystyle\prod_{k \geq 1} \frac{\left({1-x^{2kt}}\right)^{t'}}{1+x^k}\frac{1-x^{2kt}}{1-x^{kt}}\frac{(1-y^{2k}x^{2kt})^{(tz^2+z-t-1)/2}}{(1-y^kx^{kt})^{z-1}},  \text{\; if $t=2t'+1$}.
\end{numcases}

\end{Theorem}

To prove these theorems, we will need to understand precisely the behaviour of the two statistics $\delta_\lambda$ and $\varepsilon_h$ under the Littlewood decomposition. This will require some methods coming from combinatorics on words. 

From Theorem~\ref{generalisation} and Theorem~\ref{genscpair},  we are able to derive many applications, among which new signed generating functions for self-conjugate and doubled distinct partitions. However we will only highlight here consequences regarding hook-length formulas. From \eqref{eq1} is extracted in \cite[Equation (9)]{PET2} the following symplectic hook-length formula, for any positive integer $n$ and odd positive integers $t$: 
\begin{equation}
\sum_{\stackrel{ \lambda \in DD, ~ |\lambda|=2t n}{\# \mathcal{H}_t(\lambda)=2n }}\;\delta_\lambda \prod _{h \in \mathcal{H}_t (\lambda)} \frac{1}{ h\, \varepsilon_h}=\displaystyle \frac{(-1)^n}{n! 2^n t^n }. \label{eqhookdd}
\end{equation} 

Here we can state an analog in type  $\widetilde{C}\check{~}$ for self-conjugate partitions.

\begin{Theorem}\label{hooklength} We have for any positive integer $n$ and any even positive integer $t$:

\begin{equation}\sum_{\stackrel{\lambda \in SC, |\lambda|=2tn}{\# \mathcal{H}_{t}(\lambda)=2n}}\prod_{h \in \mathcal{H}_{t}(\lambda)}\frac{1}{h}= \frac{1}{n! 2^n t^n}. \label{eqhooksc}
\end{equation}
\end{Theorem}
 We would also like to emphasize we are also able to state a new companion formula for \eqref{eqhookdd} (respectively \eqref{eqhooksc}) when $t$ is even (respectively odd), but the right-hand sides are less explicit and are expressed as the coefficients of the Taylor expansion of some exponential functions (see Corollary~\ref{corfin1}).

The paper is organized as follows. In Section~\ref{section2}, we recall the definitions and notations regarding partitions, and we introduce the Littlewood decomposition. Next we prove new properties of this decomposition, and in particular we examine the behaviour of the two statistics $\delta_\lambda$ and $\varepsilon_h$. In Section~\ref{section3}, we prove our generalizations  (Theorem~\ref{generalisation} and \ref{genscpair}), from which  we derive the aforementioned consequences.  We postpone in Section~\ref{section4}  the proof of Theorem~\ref{theoremeintro}, where we introduce the notion of generalized hook length.

\section{Littlewood decomposition for integer partitions}
\label{section2}
In all this section, $t$ is a fixed positive integer.
\subsection{Definitions}\label{defs}
We recall the following definitions, which can be found in \cite{EC}. A \emph{partition} $\lambda=(\lambda_1,\lambda_2,\ldots, \lambda_\ell)$ of the integer $n \geq 0$ is a finite non-increasing sequence of positive integers whose sum is $n$. The $\lambda_i$'s are the \textit{parts} of $\lambda$, $\ell := \ell(\lambda)$ is its \textit{length}, and $n$ its \textit{weight}, denoted by $|\lambda|$. Each partition can be represented by its \textit{Ferrers diagram} as shown in Figure~\ref{fig1.1}, left. (Here we represent the Ferrers diagram in French convention.)

\begin{figure}[!h] \begin{center}
\includegraphics[scale=1.2]{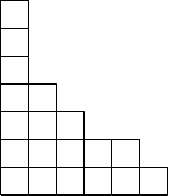}
\hspace*{1cm}
\includegraphics[scale= 1.2]{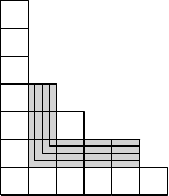}
\hspace*{1cm}
\includegraphics[scale=1.2]{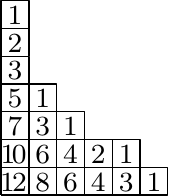}
\caption{\label{fig1.1} The Ferrers diagram of the partition $(6,5,3,2,1,1,1)$, one of its hooks, and its hook lengths.}
\end{center}
\end{figure}

For each box $v=(j,k)$ in the Ferrers diagram of $\lambda$ (with $j \in \{1,\ldots, \ell\}$ and $k \in \{1,\ldots,\lambda_j \}$), we define the \emph{hook of $v$} as the set of boxes $u$ such that either $u$ lies on the same row and on the right of $v$, or $u$ lies on the same column and above $v$. The \emph{hook length} $h:=h_v$ of $v$ is the cardinality of its hook (see Figure~\ref{fig1.1}, center). The \emph{principal diagonal} is the diagonal including the box $(1,1)$, and a box (and its hook) is \emph{principal} if it belongs to the principal diagonal.
 The \emph{Durfee square} of $\lambda$ is the greatest square included in its Ferrers diagram, the length of its side is the \emph{Durfee length}, denoted by $D(\lambda)$. 
 We define two statistics associated with the partition $\lambda$. 
 
 \begin{definition}Let $\lambda$ be a partition.  The first statistic  is denoted by $\delta_\lambda$ and defined as $ (-1)^{D(\lambda)}$. The second statistic, $\varepsilon_h$, is defined on each box of $\lambda$  as $-1$ if $h:= h_v$ is a hook length of a box $v$ strictly above the principal diagonal, and $1$ otherwise.
 \end{definition}
 
  The \emph{hook length multi-set} of $\lambda$, denoted by $\mathcal{H}(\lambda)$, is the multi-set of all hook lengths of $\lambda$.
We say that a partition $\lambda$ is a \emph{ $t$-core} if and only if no hook length of $\lambda$ is a multiple of $t$. 
Recall \cite[p. 468]{EC} that $\lambda$ is a $t$-core if and only if $\mathcal{H}(\lambda)$ does not contain the integer $t$. We denote by $\mathcal{P}$ the set of partitions and by  $\mathcal{P}_{(t)}$ the subset of $t$-cores. 
The \emph{conjugate partition} $\lambda^*$ of a partition $\lambda$ is the one whose Ferrers diagram is the reflection of that of $\lambda$ with respect to the principal diagonal.  
Recall that a partition $\lambda$ can be represented by its \emph{Frobenius coordinates} \begin{equation} \lambda=
\begin{pmatrix}
a_1 &\ldots&a_{D(\lambda)} \\ b_1 &\ldots& b_{D(\lambda)}
\end{pmatrix},
\end{equation} 
where $a_i = \lambda_i -i$ and  $b_i = \lambda_i^*-i$, for $i \in \{1, \ldots, D(\lambda)\}$. We define the set  of \emph{self-conjugate partitions} $SC$ as the set of partitions $\lambda$ such that $a_i=b_i$ for all $i \in \{1, \ldots, D(\lambda)\}$, and similarly we define the set of \emph{doubled distinct partitions} $DD$ as the set of partitions $\lambda$ such that $a_i=b_i+1$ for all $i \in \{1, \ldots, D(\lambda)\}$
 (see an example in Figure~\ref{fig9} below). We define the set of \emph{self-conjugate $t$-cores} $SC_{(t)}$ (respectively \emph{doubled distinct $t$-cores} $DD_{(t)}$) as the subset of $P_{(t)}$ made of self-conjugate (respectively doubled distinct) partitions. 

\subsection{The Littlewood decomposition}\label{section4.1}

We follow Han \cite{HAN} here. The Littlewood decomposition is a classical bijection which maps each partition to its $t$-core and $t$-quotient (see for example \cite[p. 468]{EC}).  Let $\mathcal{W}$ be the set of bi-infinite binary sequences beginning with infinitely many $0$'s and ending with infinitely many $1$'s. Each element $w$ of $\mathcal{W}$ can be represented by a sequence $(b_i)_i= \cdots b_{-2}b_{-1}b_0b_1b_2 \cdots$, but the representation is not unique. Indeed, for any fixed integer $k$ the sequence $(b_{i+k})_i$ represents $w$. The \emph{canonical representation} of $w$ is the unique sequence $(c_i)_i=\cdots c_{-2}c_{-1}c_0c_1c_2 \cdots$ such that:
\begin{equation*}
\#\{i \leq -1, c_i=1\}= \#\{i \geq 0, c_i=0\}.
\end{equation*}

We put a dot symbol ``." between the letters $c_{-1}$ and $c_0$ in the bi-infinite sequence $(c_i)_i$ when it is the canonical representation.
There is a natural one-to-one correspondence between the set of partitions $\mathcal{P}$ and $\mathcal{W}$, which we recall now. Let $\lambda$ be a partition. We encode each horizontal edge of the Ferrers diagram of $\lambda$ by $1$ and each vertical edge by $0$. Reading these $(0,1)$-encodings from top to bottom and left to right yields a binary word $u$. By adding infinitely many $0$'s to the left and infinitely many $1$'s to the right of $u$, we get an element $w=\cdots000u111\cdots \in \mathcal{W}$. The map 
\begin{equation}\label{defpsi}\psi : \lambda \mapsto w \end{equation} 
is a one-to-one correspondence between $\mathcal{P}$ and $\mathcal{W}$. The canonical representation of $\psi(\lambda)$ will be denoted by $C_\lambda$. 
\begin{figure}[h!]\begin{center}
\includegraphics[scale=1.2]{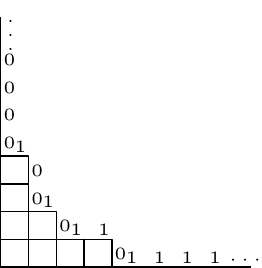}\hspace*{50pt}\includegraphics[scale=1.2]{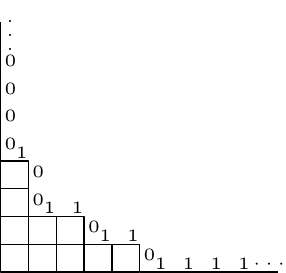}
\caption{\label{fig9} A self-conjugate and a doubled distinct partition, and their respective $(0,1)$-encodings.}
\end{center}
\end{figure}

For example (see Figure~\ref{fig9}), for $\lambda=(4,2,1,1)$, we have $u=10010110$, so that the image $\psi(\lambda)=w=\cdots00010010110111\cdots$ and $C_\lambda=\cdots0001001.0110111\cdots$.

Notice that the symbol ``." in the canonical representation of $\lambda$ corresponds in the Ferrers diagram to the north east corner of its Durfee square. The size of the latter is the number of $1$'s before the symbol ``." in the canonical representation. 

Notice that this canonical representation can be used to describe self-conjugate and doubled-distinct partitions.
Indeed, if $v$ is a finite binary word, we define $f(v)$ as the reverse word of $v$ in which we exchange the letters $0$ and $1$. Then, $\lambda$ is self-conjugate if and only if its canonical representation is of the form $\cdots00v.f(v)11\cdots$; and  $\lambda$ is doubled-distinct if and only if its canonical representation is of the form $\cdots00v.1f(v)11\cdots$.

Now we fix a partition $\lambda$, and we show how to associate a vector of partitions with $\lambda$ (involving the fixed integer $t$).

 We split the canonical representation $C_\lambda=(c_i)_i$ into $t$ sections, \emph{i.e.} we form the subsequence $v^k=(c_{it+k})_i$ for each $k \in \{0, \ldots, t-1\}$. The partitions $\lambda^0, \,\lambda^1, \ldots, \, \lambda^{t-1}$ are defined as  $\psi^{-1}(v^0), \, \psi^{-1}(v^1), \ldots, \, \psi^{-1}(v^{t-1})$, respectively. Notice that the subsequence $v^k$ is not necessarily the canonical representation of $\lambda^k$.
For each subsequence $v^k$, we continually replace the subword $10$ by $01$. The final resulting sequence is of the form $\cdots 000111\cdots$ and is denoted by $w^k$. The $t$-core of $\lambda$ is the partition $\tilde{\lambda}$ such that the $t$ sections of the canonical representation $C_{\tilde{\lambda}}$ are exactly $w^0,w^1, \ldots, w^{t-1}$ (an equivalent definition using ribbons can be found in   \cite[p. 468]{EC}). 
\begin{Theorem}[{\cite[p. 468]{EC}}]\label{littlewooddecomp}

Let $t$ be a positive integer. The \emph{Littlewood decomposition} $\Omega$, which maps a partition $\lambda$ to ($\tilde{\lambda}, \lambda^0, \lambda^1,\ldots, \lambda^{t-1}$) is a bijection such that:
\begin{itemize}
\item[(i)] $\tilde{\lambda}$ is the $t$-core of $\lambda$ and $\lambda^0, \lambda^1,\ldots, \lambda^{t-1}$ are  the partitions defined above;
\item[(ii)]$|\lambda|= |\tilde{\lambda}|+t(|\lambda^0|+ |\lambda^1|+\cdots+ |\lambda^{t-1}|)$;
\item[(iii)]$\{h/t, h \in \mathcal{H}_t(\lambda)\}=\mathcal{H}(\lambda^0)\cup\mathcal{H}(\lambda^1)\cup \cdots \cup \mathcal{H}(\lambda^{t-1})$, where this equality should be understood in terms of multi-sets. Moreover, there is a canonical one-to-one correspondence between the boxes of $\lambda$ with hook lengths which are multiples of $t$ and the boxes of $\lambda^0, \lambda^1,\ldots,\lambda^{t-1}$.
\end{itemize}

\end{Theorem}

\begin{figure}[h!]\begin{center}
\includegraphics[scale=1]{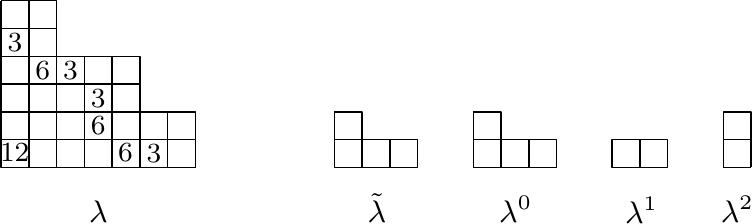}
\caption{\label{fig7}Littlewood decomposition of $\lambda=(7,7,5,5,2,2)$ with $t=3$. In the Ferrers diagram of $\lambda$, we only write the hook lengths which are integral multiples of $3$.}\end{center}
\end{figure}
\begin{Remark}\label{remark}In the proof of this theorem, properties (ii) and (iii) are consequences of the following fact, which will be useful in the sequel: the boxes of $\lambda$ are in one-to-one correspondence with ordered pairs of integers $(i,j)$ such that $i<j$, $c_i=1$ and $c_j=0$ in $C_\lambda$. Moreover the hook length of the box associated with the pair $(i,j)$ is equal to $j-i$.\end{Remark}

\subsection{Littlewood decomposition of self-conjugate partitions}

Let $t$ be a positive integer. We know from \cite[Section 7]{GKS} that the restriction of the Littlewood decomposition $\Omega$ to the set $SC$ is a bijection with the set of vectors of the form $(\tilde{\lambda},\lambda^0, \lambda^ 1, \ldots, \lambda^ {t-1}) \in SC_{(t)} \times \mathcal{P}^ {t}$, such that $\lambda^{t-i-1}=\lambda^{i*}$ for $0\leq i \leq t-1$ (where we recall that $\lambda^{i*}$ is the conjugate partition of $\lambda^{i}$). Notice that this implies that  the partition $\lambda^{t'}$ is self-conjugate when $t=2t'+1$ is odd (see Figure~\ref{exlittlewood2} for an example). 

\begin{figure}[h!]\begin{center}
\includegraphics[scale=1]{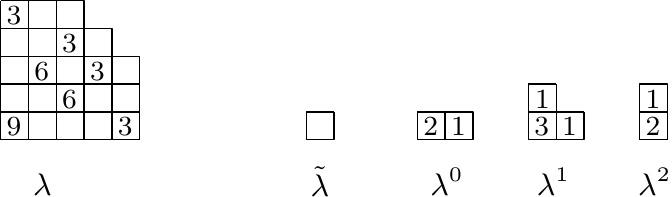}
\caption{\label{exlittlewood2}Littlewood decomposition of the self-conjugate partition $\lambda=(5,5,5,4,3)$ for $t=3$. The partition $\lambda^1$ is self-conjugate, and the partition $\lambda^2$ is the conjugate of $\lambda^0$. In the Ferrers diagram of $\lambda$, we only write the hook lengths which are integral multiples of $3$.}\end{center}
\end{figure}

To prove our two last generalizations of Theorem~\ref{theoremeintro} stated in Theorem~\ref{genscpair} below, we will need new properties of the Littlewood decomposition $\Omega$ of self-conjugate partitions, which precisely describe the behavior of the two statistics $\delta_\lambda$ and $\varepsilon_h$ when we apply $\Omega$ to a self-conjugate partition.

\begin{lemma}\label{littlewood3}
Let $t$ be a positive integer and let $\lambda$ be a self-conjugate partition. Set $(\tilde{\lambda},\lambda^0, \lambda^ 1, \ldots, \lambda^ {t-1})$ its image under $\Omega$.
The following properties hold:
\begin{itemize}
\item[(i)] If $t=2t'+1$ is odd, then $\delta_\lambda = \delta_{\tilde{\lambda}}\delta_{\lambda^{t'}}$. If $t$ is even, then $\delta_\lambda = \delta_{\tilde{\lambda}}$.
\item[(ii)] If $t=2t'+1$ id odd, let $v$ be a box in  $\lambda^{t'}$  and let $V$ be the box in $\lambda$ canonically associated with $v$ (in the sense of Remark~\ref{remark}). The box $v$ is strictly above the principal diagonal of   $\lambda^{t'}$ if and only if $V$ is strictly above the principal diagonal of  $\lambda$.
\item[(iii)] Write $t=2t'$ or $t=2t'+1$ according to the parity of $t$. Let  $i$ be in $\{0, \ldots,t'-1\}$. Let $v=(j,k)$ be a box in $\lambda^{i}$ and denote by $v^ *$ the box $(k,j)$ in $\lambda^ {i*}=\lambda^{t-1-i}$, and by $V$ and $V^ *$  the boxes in $\lambda$ canonically associated with $v$ and $v^ *$ respectively. If $V$ is strictly above (respectively strictly below)  the principal diagonal of $\lambda$, then $V^*$ is strictly below (respectively strictly above) the principal diagonal of $\lambda$. In particular, we have $\varepsilon_V \varepsilon_{V^*}=\varepsilon_v\varepsilon_{v*}=-1$.
\end{itemize}
\end{lemma}

\begin{proof}
Let $\lambda $ be a self-conjugate partition, and let $\Omega(\lambda)=(\tilde{\lambda}, \lambda^0, \lambda^1,\ldots, \lambda^{t-1})$ be its image under $\Omega$. Set $\tilde{v}:= \psi({\tilde{\lambda}})$,  and $v^i := \psi(\lambda^i)$ for $0\leq i \leq t-1$ where we recall that $\psi$ is defined in \eqref{defpsi}.

We will focus here on the case where $t$ is odd, which is the most complicated one. Indeed, in this case, the partition $\lambda^{t'}$ in the $t$-quotient is self-conjugate and requires a specific attention. When $t$ is even, all partitions in the $t$-quotient play the same role, which simplifies the proof.

Assume from now on that $t$ is odd. We prove (i) by induction on the number of boxes in the principal diagonal of $\lambda$. It is true if $\lambda$ is empty, as the $t$-core and the $t$ quotient contain empty partitions. If $\lambda$ is non-empty, we denote by $\lambda'$ the doubled distinct partition obtained by deleting the largest principal hook length of $\lambda$. Set $\Omega(\lambda'):=(\tilde{\lambda'}, \lambda^{'0 }, \lambda^{'1 },\ldots, \lambda^{'t-1} )$,  $\tilde{v'}:= \psi(\tilde{\lambda'})$, and $ v'^i:=\psi(\lambda^{'i})$ for $0\leq i \leq t-1$. By the induction hypothesis, we have $\delta_ {\lambda'}= \delta_{\tilde{\lambda'}} \delta_{\lambda^{'{t'}}}$. Deleting the largest principal hook length of $\lambda$ corresponds, in terms of words, to turn the first $1$ of $\psi(\lambda)$ into a $0$, and the last $0$ of $\psi(\lambda)$ into a $1$.
Two cases can occur.

{\bf Case 1:} the first $1$ in $\psi(\lambda)$ belongs to $v^{t'}$ (see the example in Figure~\ref{fig13} below). This implies that there are exactly $kt+t'+1$ letters between this $1$ and the symbol ``." (including this $1$), where $k$ is an integer. As $\lambda$ is self-conjugate, there are also $kt+t'+1$ letters between the symbol ``." and the last letter $0$ (including this $0$). So the last $0$ belongs also to $v^{t'}$. Turning the first $1$ into a $0$ and the last $0$ into a $1$ actually changes only $v^{t'}$ and deletes the largest principal hook of $\lambda^{t'}$. The $t$-core $\tilde{\lambda}$ does not change. So $\lambda^{'{t'}}$ is equal to the partition $\lambda^{t'}$ in which we delete the largest principal hook and $\delta_\lambda=-\delta_{\lambda'}=-\delta_{\tilde{\lambda'}} \delta_{\lambda^{'{t'}}}=\delta_{\tilde{\lambda}} \delta_{\lambda^{t'}}$.

\begin{figure}[!h]\begin{center}
\includegraphics[scale=1.2]{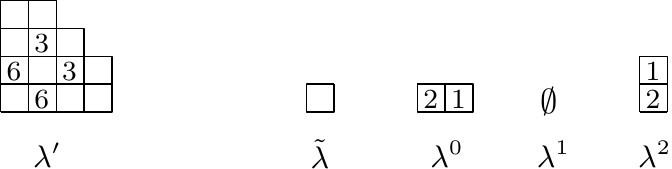}
\vspace*{0.5cm}

\includegraphics[scale=1]{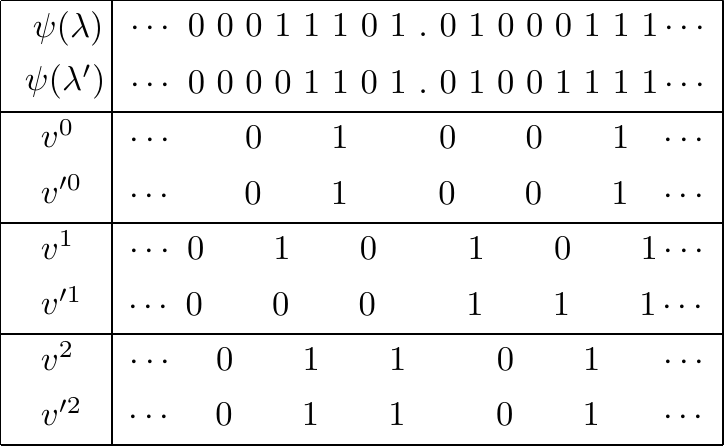}
\caption{\label{fig13}Illustration of the proof of (i) Case 1, for $\lambda=(5,5,5,4,3)$ as in Figure~\ref{exlittlewood2}, and $\lambda'=(4,4,3,2)$.}\end{center}
\end{figure}

{\bf Case 2:} the first $1$ in $\psi(\lambda)$ belongs to a word $v^i$ (with $0\leq i \leq t-1$ and $i \neq t'$) and the last $0$ in $\psi(\lambda)$ belongs to the word $v^{t-1-i}$. Turning the first $1$ into a $0$ deletes the first column in $\lambda^i$ and turning the last $0$ into a $1$ deletes the first row in $\lambda^{t-1-i}$. The partition $\lambda^{t'}$ does not change, so $\delta_{\lambda^{t'}}=\delta_{\lambda^{'{t'}}}$. We want actually to prove that the parity of the Durfee square of $\tilde{\lambda}$ is different from the parity of the one of $\tilde{\lambda'}$. By continually replacing the subword 10 by 01 in $v^i$ (respectively $v^{t-1-i}$), we obtain a subword of the form $\cdots 000111\cdots$, where the last $0$ is in position $k_1$ (respectively $k_2$), where $k_1$ and $k_2$ are integers.
By continually replacing the subword 10 by 01 in $v'^i$ (respectively $v'^{t-i}$), we also obtain a subword of the form $\cdots 000111\cdots$, but here the last $0$ is in position $k_1+1$ (respectively $k_2-1$). According to the respective distances of $k_1$ and $k_2$ to the ``." symbol in $\tilde {v}$, the Durfee square of $\tilde{\lambda'}$  increases or decreases by $1$ the Durfee square of $\tilde{\lambda}$. So $\delta_\lambda=-\delta_{\tilde{\lambda'}}\delta_{\lambda'^0}=\delta_{\tilde{\lambda}
}\delta_{\lambda^0}$. This ends the proof of $(i)$.

\medskip

We now prove $(ii)$. Let $v=(j,k)$ be a box of $\lambda^0$ and set $V$ its canonically associated box in $\lambda$. If $\lambda^{t'}$ is empty, the property is trivial. Otherwise, as $\lambda^{t'}$ is self-conjugate, we can write $v^0= \cdots 00w.f(w)11\cdots$, where $w$ is a word beginning by $1$, and as before, $f(w)$ is the reverse word of $w$ in which we exchange the letters $0$ and $1$. Assume that $v$ is strictly above the principal diagonal of $\lambda^{t'}$ and in its Durfee square (the three other cases are similar and left to the reader). Using the one-to-one correspondence between boxes and ordered pairs of integers, we can decompose $w$ in the following way: $w=w_11w_2$ where there are exactly $k-1$ occurrences of $1$ in $w_1$. We can also decompose $f(w)$ as $f(w)=w_30w_4$, where there are exactly $j-1$ occurrences of $0$ in $w_4$. As $\lambda^{t'}$ is self-conjugate and as $j>k$,  $f(w_1)$ is a suffix of $w_4$. This situation also holds in the canonical representation of $\psi(\lambda)$. We can write $\psi(\lambda)=\cdots 00u_11u_2.u_30u_411\cdots$, where the pair made of the first $1$ after $u_1$ and the first $0$ after $u_3$ corresponds to the box $V$. The word $f(u_1)$ is a suffix of $u_4$, and $u_1$ contains strictly less $1$'s than the number of $0$'s in $u_4$, so the box $V$ is strictly above the principal diagonal of $\lambda$.

\medskip

We finally prove (iii). Notice first that for a parity reason, the boxes $V$ and $V^*$ can not belong to the principal diagonal of $\lambda$. Assume without loss of generality that $k <j$ (and so $V$ is strictly above the principal diagonal in $\lambda$) . We can write $v^i=\cdots00w_11w_20w_311\cdots$, where $w_1$ begins by $1$, $ w_3$ ends by $0$, and there are $k-1$ occurrences of $1$ in $w_1$ and $j-1$ occurrences of $0$ in $w_3$ (where here the $1$ and the $0$ before and after $w_2$ correspond to the box $v$). As $\lambda^{t-1-i}=\lambda^{i*}$, we can write $v^{t-1-i}=f(v^i)=\cdots00u_11u_20u_311\cdots$ where $u_1$ begins by $1$, $ u_3$ ends by $0$, and there are $j-1$ occurrences of $1$ in $u_1$ and $k-1$ occurrences of $0$ in $u_3$. Here, the $1$ and the $0$ before and after $u_2$ correspond to the box $v^*$. These informations allow us to describe $\psi(\lambda)$ in the following way:
$ \psi(\lambda)= \cdots00x_11x_21x_30x_40x_511\cdots$, where the $1$ after $x_1$ comes from the $k^{th}$ occurrence of $1$ in $v^i$, the $1$ after $x_2$ comes from the $i^{th}$ occurrence of $1$ in $v^{t-1-i}$, the $0$ after $x_3$ comes from the $i^{th}$ occurrence of $0$ reading $v^{i}$ from right to left and the $0$ after $x_4$ comes from the $k^{th}$ occurrence of $0$ reading $v^{t-1-i}$ from right to left. The box $V$ corresponds to the ordered pair given by the $1$ after $x_1$ and the $0$ after $x_3$. Therefore there is at least one more occurrence of $0$ in $x_40x_5$ than the number of occurrences of $1$ in $x_1$, so the box $V$ is strictly above the principal diagonal in $\lambda$. The box $V^*$ corresponds to the ordered pair given by the $1$ after $x_2$ and the $0$ after $x_4$, so $V^*$ is strictly below the principal diagonal in $\lambda$.
\end{proof}

The properties in Lemma~\ref{littlewood3} can be checked in the example of  Figure~\ref{exlittlewood2}.

The first property of the above lemma allows us to compute the following signed generating function of self-conjugate $t$-cores. 

\begin{lemma}\label{fungensign} Let $t$ be a positive integer. We have
\begin{numcases}
{\hspace*{-15pt}\sum_{\lambda \in SC_{(t)}} \delta_\lambda \, x^ {|\lambda|}=}\displaystyle\prod_{k \geq 1} \frac{1-x^k}{1-x^{2k}}\left({1-x^{2kt}}\right)^{t'}, &\text{ if $t=2t'$,~~}
\\\displaystyle
\prod_{k \geq 1} \frac{1-x^k}{1-x^{2k}}\left({1-x^{2kt}}\right)^{t'}\frac{1-x^{2kt}}{1-x^{kt}}, &\text{ if $t=2t'+1$}\label{seriegen11}.~~
\end{numcases}
\end{lemma}

\begin{proof}
 We only prove here the case $t=2t'+1$, the other case $t=2t'$ being similar. By the Littlewood decomposition of self-conjugate partitions and Lemma~\ref{littlewood3}~(i) we obtain:
 \begin{equation}
\sum_{\lambda \in SC} \delta_{\lambda}\, x^{|\lambda|/2}=\sum_{\tilde{\lambda} \in SC_{(t)}} \delta_{\tilde{\lambda}}\, x^{|\tilde{\lambda}|} \times \sum_{\lambda^{t'} \in SC} \delta_{\lambda^{t'}}\, x^{t|\lambda^{t'}|} \times \left( \sum_{\lambda \in \mathcal{P}} x ^{2t|\lambda|} \right)^{t'}. \label{proofseriegen}
 \end{equation}
By enumerating self-conjugate partitions according to their principal hook length, we also derive:
\begin{equation*}
\sum_{\lambda^{t'} \in SC} \delta_{\lambda^{t'}}\, x^{t|\lambda^{t'}|}= \prod_{k \geq 1}(1-x^{(2k+1)t}) = \prod_{k \geq 1}\frac{1-x^{kt}}{1-x^{2kt}}.
\end{equation*}
As the generating function of integer partitions is well-known, formula \eqref{seriegen11} follows.
\end{proof}

\subsection{Littlewood decomposition of doubled distinct partitions}\label{section4.2}
Let $t$ be a positive integer.
It is already known (see \cite[Bijection 3]{GKS}) that the restriction of $\Omega$ to the set of doubled distinct partitions gives a bijection with the set of vectors ($\tilde{\lambda}, \lambda^0, \lambda^1,\ldots,\lambda^{t-1}) \in DD_{(t)} \times DD \times \mathcal{P}^{t-1}$ such that $\lambda^{t-i} = \lambda^{i*}$ for $1 \leq i \leq t-1$ (where we recall that $\lambda^{i*}$ is the conjugate of $\lambda^i$). Note that if $t=2t'$ is even, then the partition $\lambda^{t'}$ is self-conjugate.
This property can be checked on Figure~\ref{fig7}, for $\lambda=(8,7,6,5,3,2,1)\in DD$ and $t=3$.

 To prove our generalization of \eqref{eqtheoremeintro}, namely Theorem~\ref{generalisation}, we will need new properties of the Littlewood decomposition for doubled distinct partitions. The following Theorem express these properties, the easier $t$ odd case is already known (see \cite[Lemma 4.2]{PET2}). The $t$ even case is similar to Lemma~\ref{littlewood3}, therefore we do not detail the proof.

\begin{lemma}\label{littlewood}
Let $t$ be a positive integer and write $t=2t'$ or $t=2t'+1$ according to the parity of $t$. Let $\lambda $ be a doubled distinct partition, and let $(\tilde{\lambda}, \lambda^0, \lambda^1,\ldots, \lambda^{t-1})$ be its image under $\Omega$. The following properties hold:
\begin{itemize}
\item[(i)] $\delta_ \lambda= \delta_{\tilde{\lambda}} \delta_{\lambda^0}$ if $t$ is odd and $\delta_ \lambda= \delta_{\tilde{\lambda}} \delta_{\lambda^0} \delta_{\lambda^{t'}}$ if $t$ is even.
\item[(ii)] Let $v$ be a box of $\lambda^0$ and let $V$ be its canonically associated box in $\lambda$ (in the sense of Remark~\ref{remark}). The box $v$ is strictly above the principal diagonal of $\lambda^0$ if and only if $V$ is strictly above the principal diagonal of $\lambda$.

\item[(iii)] Let $v=(j,k)$ be a box of $\lambda^i$, with $1\leq i\leq t'$. Denote by $v^*=(k,j)$ the box of $\lambda^{2t'+1-i}=\lambda^{i*}$ and by $V$ and $V^*$ the boxes of $\lambda$ canonically associated with $v$ and $v^*$. If $V$ is strictly above (respectively below) the principal diagonal in $\lambda^i$, then $V^*$ is strictly below (respectively above) the principal diagonal in $\lambda$.

\end{itemize}
\end{lemma}

\begin{figure} [h!]\begin{center}
\includegraphics[scale=1.2]{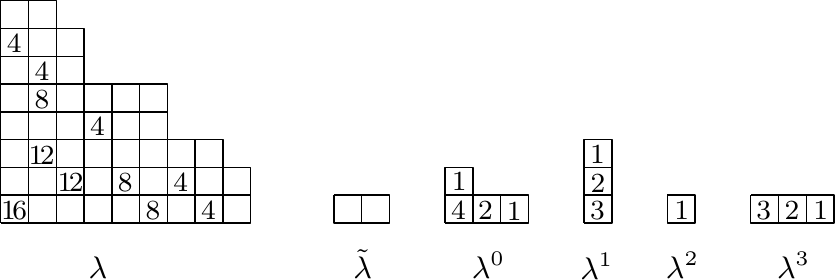}
\caption{\label{fig11}Illustration of Lemma~\ref{littlewood}  for $\lambda= (9,9,8,6,6,3,3,2)$ and $t=4$. We have $t'=2$ and $ \delta_\lambda= \delta_{\tilde{\lambda}}\delta_{\lambda^0}\delta_{\lambda^2}$.}\end{center}
\end{figure}

Property (i) of the previous lemma allows us to compute the following signed generating function of doubled distinct $t$-cores, which is surprisingly simpler than the unsigned one given in \cite{GKS}. As the proof relies on the same arguments as the one of Lemma~\ref{fungensign}, it is omitted. 

\begin{lemma}\label{seriegen}
Let $t=2t'$ be an even positive integer. The following equality holds:
\begin{equation}
\sum_{\tilde{\lambda} \in DD_{(t)}} \delta_{\tilde{\lambda}}\, x^{|\tilde{\lambda}|/2}
=\prod_{k \geq 1} \frac{(1-x ^k) (1-x^{kt})^{t'-1}}{1-x ^{kt'}}.
\end{equation}

\end{lemma}

Notice that the $t$ odd case is already known, but not useful here.

\section{Generalizations of character formulas}\label{section3}

\subsection{Proof of Theorem~\ref{generalisation}}\label{section4.3}
Here we prove Theorem~\ref{generalisation}, which will be seen to generalize \eqref{eqtheoremeintro}, and we derive a first application.

\begin{proof}[Proof of Theorem~\ref{generalisation}.]
Let $t=2t'$ be a positive integer and let $\lambda$ be a doubled distinct partition. We  first transform the expression \begin{equation}\label{eqrhs}\displaystyle \delta_{\lambda} x^{|\lambda|/2} \prod_{ h \in \mathcal{H}_t(\lambda)}\left( y -\frac{ytz}{\varepsilon_h ~ h}\right)\end{equation} by using the Littlewood decomposition of $\lambda$. Set $\Omega(\lambda)=(\tilde{\lambda}, \lambda^0, \lambda^1,\ldots, \lambda^{t-1})$. According to Lemma~\ref{littlewood}~(i), we have $\delta_ \lambda= \delta_{\tilde{\lambda}} \delta_{\lambda^0}\delta_{\lambda^{t'}}$.

Let $B_i$ be the multi-set of hook lengths in $\mathcal{H}_{t}(\lambda)$ coming from the boxes of $\lambda$ which correspond to the ones of $\lambda^i$, for $0\leq i\leq t-1$. 
According to Theorem~\ref{littlewooddecomp} (iii) and Lemma~\ref{littlewood}~(ii), we have:
\begin{equation}
\prod_{ h \in B_0} \left(y -\frac{ytz}{\varepsilon_h ~ h}\right)= \prod_{h \in \mathcal{H}(\lambda^0)} \left(y-\frac{yz}{\varepsilon_h ~h} \right).\label{eq47}
\end{equation}
 Similarly, according to Theorem~\ref{littlewooddecomp} (iii) and Lemma~\ref{littlewood}~(iii) applied for $i=t'$, we have:
\begin{equation}
\prod_{ h \in B_{t'}} \left(y -\frac{ytz}{\varepsilon_h ~ h}\right)= \prod_{h \in \mathcal{H}(\lambda^{t'})} \left(y-\frac{yz}{\varepsilon_h ~h} \right).\label{eq47a}
\end{equation}
 
 Let $v=(j,k)$ be a box of $\lambda^i$, with $1\leq i\leq t'-1$. Denote by $v^*=(k,j)$ the box of $\lambda^{2t'+1-i}=\lambda^{i*}$ and by $V$ and $V^*$ the boxes of $\lambda$ canonically associated with $v$ and $v^*$. 
By Lemma~\ref{littlewood}~(iii), one of the two boxes $V$ and $V^*$ is strictly below the principal diagonal of $\lambda$, and the other is strictly above the principal diagonal of $\lambda$. So we have:
\begin{equation*}\label{eqhV}
\left(y -\frac{ytz}{\varepsilon_{h_V} ~ h_V}\right)\left(y -\frac{ytz}{\varepsilon_{h_ {V^*}} ~ h_{V^*}}\right)=y^2 -\left(\frac{ytz}{ ~ h_V}\right)^2= y^2 -\left(\frac{yz}{ h_v}\right)^2,
\end{equation*}
where the last equality follows from $h_V= th_v$ according to  Theorem~\ref{littlewooddecomp} (iii).
Multiplying this over all boxes $V$ in $B_i$ gives:
\begin{equation}
\prod_{ h \in B_i\cup B_{t-i}} \left(y -\frac{ytz}{\varepsilon_h ~ h}\right)= \prod_{ h \in \mathcal{H}(\lambda ^ i)} \left(y^2 -\left(\frac{yz}{h}\right)^2 \right).\label{eq48}
\end{equation}
Using \eqref{eq47}--\eqref{eq48}, and Theorem~\ref{littlewooddecomp} (ii),   we can rewrite \eqref{eqrhs} as follows:
\begin{multline*}
\delta_{\lambda} x^{|\lambda|/2} \prod_{ h \in \mathcal{H}_t(\lambda)}\left( y -\frac{ytz}{\varepsilon_h ~ h}\right)=\delta_{\tilde{\lambda}} x^{|\tilde{\lambda}|/2}~~\times~~\delta_{\lambda ^0} x^{t|\lambda^0|/2} \prod_{h \in \mathcal{H}(\lambda^0)} \left(y-\frac{yz}{\varepsilon_h ~h} \right)\\\times \prod_{i=1}^{t'-1}  x^{t|\lambda^i|} \prod_{h \in \mathcal{H}(\lambda^i)} \left(y^2 -\left(\frac{yz}{h}\right)^2 \right) \times \delta_{\lambda ^{t'}} x^{t|\lambda^0|/2} \prod_{h \in \mathcal{H}(\lambda^{t'})} \left(y-\frac{yz}{\varepsilon_h ~h} \right).
\end{multline*}
We sum this over all doubled distinct partitions $\lambda$. The left-hand side becomes the one of \eqref{eq2}, while the right-hand side can be written as a product of four terms.  The first one is 
\begin{equation}
\sum_{\tilde{\lambda} \in DD_{(t)}}\delta_{\tilde{\lambda}} x^{|\tilde{\lambda}|/2}=\prod_{k \geq 1} \frac{(1-x ^k) (1-x^{kt})^{t'-1}}{1-x ^{kt'}}, \label{eqn1}
\end{equation}
through Lemma~\ref{seriegen}. The second is 
\begin{equation}
\sum_{\lambda^0 \in DD} \delta_{\lambda ^0} x^{t|\lambda^0|/2} \prod_{h \in \mathcal{H}(\lambda^0)} \left(y-\frac{yz}{\varepsilon_h ~h} \right)= \prod_{k \geq 1}(1-y^{2k}x^{tk})^{(z^2-3z+2)/2},
\end{equation}
by \eqref{eqtheoremeintro} applied with $x$ replaced by $y^2x^t$ and $2z+2$ replaced by $t$. The third is
\begin{equation}
\left(\sum_{\lambda \in \mathcal{P}}  x^{t|\lambda^i|} \prod_{h \in \mathcal{H}(\lambda^i)} \left(y^2 -\left(\frac{yz}{h}\right)^2 \right)\right)^{t'-1}= \prod_{k \geq 1}(1-y^{2k}x^{tk})^{(z^2-1)(t'-1)}, \label{eqn2}
\end{equation}
according to the classical Nekrasov--Okounkov formula \cite[Equation (1.3)]{HAN} applied with $x$ replaced by $y^2x^t$ and $2z+2$ replaced by $t$. The last is 
\begin{equation}
\sum_{\lambda^{t'} \in SC} \delta_{\lambda ^{t'}} x^{t|\lambda^0|/2} \prod_{h \in \mathcal{H}(\lambda^{t'})} \left(y-\frac{yz}{\varepsilon_h ~h} \right) = \left(\prod_{k \geq 1} \frac{(1-y^{2k}x^{tk})^{z/2+1}}{1-x^{t'k}}\right)^ {z-1},\label{eqn3}
\end{equation}
by Theorem~\ref{genscpair} applied with $x$ replaced by $y^2x^t$ and $z$ replaced by $z/2$.

\end{proof}

Notice that when $y=0$ in Theorem~\ref{generalisation}, we recover Lemma~\ref{seriegen}.

\begin{corollary}[$z =-b/y$, $y \rightarrow 0$ in Theorem~\ref{generalisation}] \label{cor6} We have for any even positive integer $t=2t'$:
\begin{equation}
{\sum_{\lambda \in DD}  \delta_\lambda\, x^{|\lambda|/2} \prod_{ h \in \mathcal{H}_t(\lambda)} \frac{bt}{ h\, \varepsilon_h} =}\exp(-bx^{t'}-tb^2x^{t}/2)\prod_{k \geq 1} \frac{(1-x^k)(1-x^{kt})^{t'-1}}{1-x^{kt'}}.~~~~~~\label{eqexp}
\end{equation}
\end{corollary}

\begin{proof}
By setting $z =-b/y$, and letting $y \rightarrow 0$ in Theorem~\ref{generalisation}, the left-hand side becomes exactly the one of \eqref{eqexp}. The limit on the right-hand side can be computed through the classical 
\begin{equation}
\prod_{k \geq 1}\frac{1}{1-x^k}= \exp\left(\sum_{k \geq 1} \frac{x^k}{k(1-x^k)} \right),
\end{equation} however we omit the details here.
\end{proof}

\subsection{Proof of Theorem~\ref{genscpair}}
By using the properties of the Littlewood decomposition of self-conjugate partitions given in Lemma~\ref{littlewood3}, we can prove two generalizations of the Nekrasov--Okounkov formula in type $\widetilde{C}\check{~}$ (namely \eqref{eqgenccheck2}).

\begin{proof}[Proof of Theorem~\ref{genscpair}]
The proof is similar to the one of Theorem~\ref{generalisation}, and we do not give all the details here. We only emphasize that, at the final step, when we sum over self-conjugate partitions, the right-hand side can be written as a product of a smaller number of terms. In the $t$ odd case, we obtain three terms, one being similar to the left-hand side of \eqref{eqn1} (here the sum ranges over self-conjugate partitions), the second being similar  to the left-hand side of \eqref{eqn2} and the third being similar to the leftt-hand side of \eqref{eqn3}. The first term can be computed through Lemma~\ref{fungensign}, the others in the same way as in the proof of Theorem~\ref{generalisation}. In the $t$ even case, only the two first terms appear.
\end{proof}

Notice that when we set $y=0$ in Theorem~\ref{genscpair}, we recover Lemma~\ref{fungensign}, and when we set $y=t=1$ in Theorem~\ref{genscpair}, we recover \eqref{eqgenccheck2}. So Theorem~\ref{genscpair} unifies the Macdonald identities generalized by Theorem~\ref{theoremeintro} and the signed generating function of self-conjugate $t$-cores. Similarly to Corollary \ref{cor6}, we have the following consequence.

\begin{corollary}[$z=-b/y$, $y \to 0$ in Theorem~\ref{genscpair}]\label{funexppair}
We have:
\begin{multline*}
\sum_{\lambda \in SC} \delta_\lambda \, x^ {|\lambda|} \prod_{h \in \mathcal{H}_t(\lambda)}\frac{bt}{h \, \varepsilon_h}
\end{multline*}
\vspace*{-3mm}
 \begin{numcases}{=}\displaystyle\exp(-b^2t'x^{2t})\prod_{k \geq 1} \frac{1-x^k}{1-x^{2k}}\left({1-x^{2kt}}\right)^{t'}, \text{ if $t=2t'$;\label{eqfunexppair}}
\\ \displaystyle \exp(-bx^t-tb^2x^{2t}/2)\prod_{k \geq 1} \frac{1-x^k}{1-x^{2k}}\left({1-x^{2kt}}\right)^{t'}\frac{1-x^{2kt}}{1-x^{kt}}, \text{ if $t=2t'+1$.~~~~~~}\label{eqfunscimpair}
\end{numcases}

\end{corollary}

\subsection{Hook formulas}
In this section, we prove our generalized hook formula in  Theorem~\ref{hooklength}.

\begin{proof}[Proof of the hook formula \eqref{eqhooksc}]
To prove \eqref{eqhooksc}, we extract the coefficients of $b^{2n}x^{tn}$ in \eqref{eqfunexppair}. We then have to simplify signs. Indeed, extracting the coefficient of  $b^{2n}x^{2tn}$ on both sides of~\eqref{eqfunexppair} gives:
\begin{equation}
\sum_{\stackrel{\lambda \in SC, |\lambda|=tn}{\# \mathcal{H}_{t}(\lambda)=2n}}\delta_\lambda\prod_{h \in \mathcal{H}_{t}(\lambda)}\frac{t}{h\,\varepsilon_h}= \frac{(-t')^n}{n!}.
\end{equation}

Note that when $\lambda$ is self-conjugate, all its principal hook lengths are odd. As the sum of all the principal hook lengths is equal to  $|\lambda|=2n$, it is even. Therefore there is an even number of principal hook lengths, and  $\delta_\lambda=1$. As $t$ is even, none of the principal hook lengths of $\lambda$  is an integral multiple of $t$. So, if a box $(j,k)$ of $\lambda$ has a hook length which is an integral multiple of $t$, the box $(k,j)$ is distinct from the box $(j,k)$. Moreover, as $\lambda$ is self-conjugate, the hook length of $(k,j)$ is also a multiple of $t$. Summarizing, we have $\varepsilon_{(j,k)}\varepsilon_{(k,j)} =-1$, and besides $\displaystyle \prod_{h \in \mathcal{H}_{t}(\lambda)} \varepsilon_h = (-1)^n$. 
\end{proof}

Unlike the case of \eqref{eqfunexppair}, trying to extract coefficients in equations~\eqref{eqexp} and \eqref{eqfunscimpair} to obtain hook formulas as explicit as \eqref{eqhookdd} and \eqref{eqhooksc} seems hopeless. Nevertheless, we obtain the following results.

\begin{corollary}We have: \label{corfin1}
\begin{equation}
\sum_{\stackrel{\lambda \in DD, |\lambda|=tn}{\#\mathcal{H}_t(\lambda)=n}}\delta_\lambda\prod_{h \in \mathcal{H}_t(\lambda)}\frac{t}{h\,\varepsilon_h}=[x^{t'n}]\exp(-x^{t'}-tx^{t}/2) \text{ if $t=2t'$;}
\end{equation}\vspace*{-0pt}
and 
\begin{equation}
\sum_{\stackrel{\lambda \in SC, |\lambda|=tn}{\#\mathcal{H}_t(\lambda)=n}}\delta_\lambda\prod_{h \in \mathcal{H}_t(\lambda)}\frac{t}{h\,\varepsilon_h}=[x^{tn}]\exp(-x^{t}-tx^{2t}/2), \text{ if $t=2t'+1$.}
\end{equation}
\end{corollary}

\section{A Nekrasov--Okounkov formula in type $\widetilde{C}\check{~}$}\label{section4}

We prove in this section a Nekrasov--Okounkov type formula, for the affine type . Recall that in this type, one of the Macdonald formulas can be written in the following form \cite[p. 137 (6)(b)]{ARS}, for all integers $t\geq 2$:
\begin{equation}\label{eqCcheck}
\left(\frac{\eta(x^2)^{t+1}}{\eta(x)}\right)^{2t-1}=c_2 \sum_{{\bf v}} x^{\|v\|^2/8t} \prod_{i<j} (v_i^2-v_j^2),
\end{equation}
where  $c_2$ is an explicit constant, and where the sum ranges over the  $t$-tuples ${\bf v} := (v_1, \ldots, v_t) \in \mathbb{Z}^t$ such that $v_i =2i-1 \mod 4t$. We show here the generalization of the previous formula, namely Theorem~\ref{theoremeintro}.

The idea of the proof is the following: starting from \eqref{eqCcheck}, we apply a bijection to rewrite the right-hand side of this identity in terms of self-conjugate $t$-cores and next simply in terms of self-conjugate partitions. Finally, an argument of polynomiality will allow us to conclude. 

Let us precise the organization of this section. In Section~\ref{sectiontcore}, we introduce the notion of generalized hook length and its properties. In Section~\ref{2tcoreauto}, we study a bijection due to Garvan--Kim--Stanton that we use to complete the proof of Theorem~\ref{theoremeintro} in Section~\ref{sectiontest}. Finally, we study consequences of this Theorem in Section~\ref{section4.4}.

\subsection{Self-conjugate $t$-cores and generalized hook lengths}\label{sectiontcore}
 
In all this section, $t$ is a fixed integer.

Recall that the definitions of  $t$-cores and self-conjugates partitions have been given in Section~\ref{defs}. We study more precisely in this section the properties of such partitions.

We can first observe the following facts. The set of principal hook lengths of a self-conjugate partition contains only odd positive integers. Moreover, if $\Delta$ is a finite set of odd positive integers, then there is a unique self-conjugate partition $\lambda$ such that $\Delta$ is the set of principal hook lengths of $\lambda$.

In order to study the self-conjugate partitions (and $t$-cores), we need to extend the notion of principal hook length.

\begin{definition}
Let $\lambda$ be a self-conjugate partition and let  $(i,j)$ be a pair of integers such that $1\leq i,j\leq \lambda_1$. We define the  \emph{generalized hook length} $h_{i,j}$ of the box $(i,j)$ as:
\begin{equation*}
h_{i,j} = \lambda_i-i+\lambda_j-j+1.
\end{equation*}
\end{definition}

\begin{figure}[h!]\begin{center}
\includegraphics[scale=1.4]{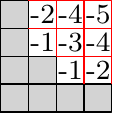}
\caption{\label{fig10}The generalized hook lengths of the boxes which do not belong to the partition $\lambda:=(4,2,1,1)$. The boxes in grey belong to the partition; the red lines outline the other boxes.}
\end{center}
\end{figure}

Notice that, as $\lambda$ is self-conjugate, if a box  $(i,j)$ belongs to the Ferrers diagram of $\lambda$, then  the notions of \emph{hook length} and \emph{generalized hook length} of $(i,j)$ coincide. If the box $(i,j)$  does not belong to the Ferrers diagram of $\lambda$, a direct calculation shows that $h_{i,j}$ is the opposite of the number of boxes $u$  such that  $u$ does not belong to the Ferrers diagram of  $\lambda$ and such that either  $u$ lies on the same line and on the left of $(i,j)$, or $u$ lies on the same column and below $(i,j)$. Notice also that, as $\lambda$ is self-conjugate, all of its  \emph{principal generalized hook lengths}   (the hook lengths of the boxes $(i,i)$ with $1\leq i \leq \lambda_1$) are odd. Moreover, we can prove the following lemma.

\begin{lemma}\label{lemme+-}
Let $\lambda$ be a self-conjugate partition and consider   $\Delta^g$ its set of principal generalized hook length. Let  $i$ be an odd integer such that $1\leq i \leq 2\lambda_1-1=h_{1,1}$. The two following properties are satisfied:
\begin{itemize}
\item[(i)] either $i$ or $-i$ belongs to $\Delta^g$,
\item[(ii)] if $i$ belongs to $\Delta^g$ then $-i$ does not belong to $\Delta^g$.
\end{itemize}
\end{lemma}

\begin{proof}
We first prove (ii). Let  $i$ be an odd integer such that $1\leq i \leq \lambda_1$. Assume that $i$ and $-i$ both belong to $\Delta^g$, and set $(j_1,j_1)$ (respectively $(j_2,j_2)$)  the principal box with generalized hook length equal to $i$ (respectively $-i$). Then the generalized hook length of the box $(j_1, j_2)$ is equal to:
\begin{equation*}
h_{j_1, j_2}= \lambda_{j_1}-j_1+\lambda_{j_2}-j_2+1= \frac{h_{j_1,j_1}+h_{j_2,j_2}}{2}=0.
\end{equation*}
This is a contradiction with the fact that $ h_{j_1, j_2}$ is the cardinal (or the opposite of the cardinal) of a non-empty set. Thus (ii) is satisfied.

There are in $\Delta^g$ exactly $\lambda_1$ odd integers (by definition of a generalized hook length), which are all different as the generalized hook length strictly decrease along the principal diagonal, which are all in absolute value smaller or equal to  $h_{1,1}$, and these integers satisfy the property (ii). So, according to the pigeonhole principle, either $i$ or $-i$ belongs to $\Delta^g$.
\end{proof}

For example, for the self-conjugate partition  $\lambda=(4,2,1,1)$ of Figure~\ref{fig10}, the set  $\{7,1,-3,-5\}$ of its principal generalized hook length satisfies the properties of Lemma~\ref{lemme+-}.

\subsection{A bijection between self-conjugate $2t$-cores and vectors of integers}\label{2tcoreauto}
In this section, we introduce a  bijection involving  self-conjugate $2t$-cores.

\begin{definition}\label{defdeltaiauto2}
Let $\lambda$ be a self-conjugate $2t$-core. For all $i\in \{1, \ldots, t\}$, we set:
\begin{equation}\label{defdeltaiauto}
\Delta_{2i-1} := \max \left(\{h \in \Delta, h \equiv \pm (2i-1)-2t \mod 4t\} \cup \{2i-1-2t\}\right).
\end{equation}
\end{definition}

\begin{figure}[h!]\begin{center}
\includegraphics[scale=1.2]{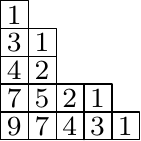}
\caption{\label{fig11n} The self-conjugate $6$-core  $(5,4,2,2,1)$. We have $\Delta_1 = 5$, $\Delta_3=9$, and $\Delta_5=-1$}.
\end{center}

\end{figure}

It is not clear at this step that the set $\displaystyle\bigcup_{i=1}^{t} \{\Delta_{2i-1}\}$ characterizes the $2t$-core $\lambda$. The aim of this section is to prove this affirmation. To this purpose, we need a classical bijection due to Garvan--Kim--Stanton \cite{GKS}, between (general) $t$-cores and vectors of integers, that we recall now.

Let $\lambda$ be a $t$-core, we define the vector $ \phi(\lambda):=(n_0, n_1,\ldots, n_{t-1})$ as follows. We label the box $(i,j)$ of $\lambda$ by $(j-i)$ modulo $t$. We also label the boxes in the (infinite) column 0 in the same way, and we call the resulting diagram the \emph{extended t-residue diagram} (see Figure~\ref{fig3} below). A box is called \emph{exposed} if it is at the end of a row of the extended $t$-residue diagram. The set of boxes $(i,j)$ of the extended $t$-residue diagram satisfying $t(r-1)\leq j-i<tr$ is called a \emph{region} and labelled $r$. We define $n_i$ as the greatest integer $r$ such that the region labelled $r$ contains an exposed box with label $i$. 
\begin{figure}[h!]\begin{center}
\includegraphics[scale=0.75]{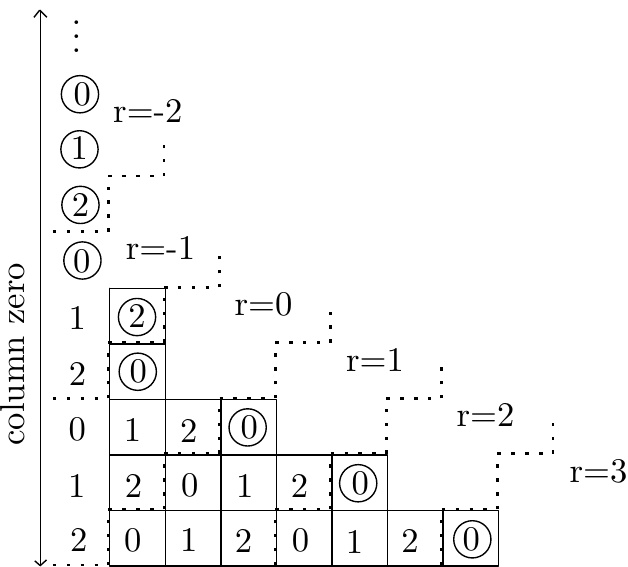}
\caption{\label{fig3}The extended 3-residue diagram of the $3$-core $\lambda=(7,5,3,1,1)$. The exposed boxes are circled.}
\end{center}
\end{figure}
\begin{Theorem}[\cite{GKS}]\label{GKS}
The map $\phi$ is a bijection between $t$-cores and vectors of integers $ {\bf n}=(n_0,n_1,\ldots, n_{t-1}) \in \mathbb{Z}^t$, satisfying $n_0+ \cdots +n_{t-1}=0$, such that: \begin{equation*} |\lambda | = \frac{t\|{\bf n}\|^2}2 + {\bf b \cdot n}= \frac{t}{2}\sum_{i=0}^{t-1}n_i^2+ \sum_{i=0}^{t-1}in_i,\end{equation*}
where~ ${\bf b } := (0,1,\ldots,t-1)$, $\|{\bf n}\|$ is the euclidean norm of ${\bf n}$, and ${\bf b \cdot n}$ is the scalar product of ${\bf b}$ and ${\bf n}$.
\end{Theorem}

We now examine the effect of this bijection to self-conjugate $2t$-cores. Let $\lambda$ be a self-conjugate $2t$-core, and set  $\Delta$ its set of principal hook lengths and $\phi(\lambda) := (m_1, \ldots, m_{2t})$ its image under $\phi$. We define $\phi_1(\lambda)$ as \begin{equation*}\phi_1(\lambda):= (m_{t+1}, m_{t+2}, \ldots, m_{2t}).
\end{equation*}

\begin{Theorem} Let $t \geq 1$ be an integer and set the vector ${\bf d}:=(1,3,\ldots, 2t-1)$.  The map $\phi_1$  defined as $\phi_1(\lambda):= {\bf n } = (n_1, \ldots, n_t)$ is a bijection between $SC_{(2t)}$ and $\mathbb{Z}^t$   satisfying:
\begin{equation}\label{eqlambdadelta}
|\lambda|= 2t\|{\bf n}\|^2+{\bf d \cdot n}= 2t\sum_{i=1}^t\left(n_i^2+(2i-1)n_i\right).
\end{equation}
Moreover, the following relation holds for any integer $i \in \{1, \ldots, t\}$:
\begin{equation}\label{eqnicelta}
\sigma_i(4tn_{i}+2i-1)= 2t+ \Delta_{2i-1},
\end{equation}
where $\sigma_i$ is equal to $1$ if $n_i \geq 0$ and to $-1$ otherwise.
\end{Theorem}

\begin{proof}
The fact that  $\phi_1(\lambda)$ is a bijection is (almost) already known, as it is a direct consequence of  \cite[Section 7]{GKS}, and the equality \eqref{eqlambdadelta} comes from \cite[(7.4)]{GKS}. Equation \eqref{eqnicelta} comes from \cite[Theorem 3.2]{PET2} (applied with $t+1$ replaced by $2t$ and a empty doubled distinct partition). Recall that the proof of this theorem comes from a precise study of the action of the bijection   $\phi$ on self-conjugate partitions.
\end{proof} 

In  Figure~\ref{fig11n}, we see that the $6$-core has for image under $\phi_1$ the vector $(-1,1,0)$, and the relations \eqref{eqnicelta} are satisfied.
 
\begin{Remark}\label{remcle}
There are three consequences of the previous theorem. Let $\lambda$ be a self-conjugate $2t$-core  and set $\Delta$ its set of principal hook lengths.
\begin{itemize}
\item[(i)] There can not be in $\Delta$ both an integer equal to   $2i-1-2t \mod 4t$ and an integer equal to $-(2i-1)-2t \mod 4t$.
\item[(ii)] If $h > 4t$ belongs to $\Delta$, then $h-4t$ also belongs to $\Delta$.
\item[(iii)] If a finite subset of  $\mathbb{N}$ satisfies the two former properties (i) and (ii) and does not contain any even integer, then it is the set of principal hook lengths of a self-conjugate and doubled distinct $2t$-core.
\end{itemize}
\end{Remark}

By writing $ v_i =4t n_i +2i-1$ for all integers $i \in \{1, \ldots, t\}$, we can replace the sum in Macdonald formula~\eqref{eqCcheck} by a sum ranging over self-conjugate   $2t$-cores, thanks to our bijection $ \phi_1$ and to \eqref{eqnicelta}. Using the equality
\begin{equation*}
\frac{\|v\|^2}{8t}= |\lambda| +\frac{(2t-1)(2t+1)}{24},
\end{equation*}
coming from \eqref{eqlambdadelta}, we eliminate the factor of modularity  $x^{1/24}$ into Macdonald formula. This leads to 
\begin{equation}\label{eqreecrire}
\left(\prod_{i \geq 1} \frac{(1-x^{2i})^{t+1}}{1-x^i}\right)^{2t-1}= c_2 \sum_{\lambda \in SC_{(2t)}} x^{|\lambda|} \,\prod_{i<j} \left((2t+\Delta_i) ^2- (2t+\Delta_j)^2 \right),
\end{equation}
where the product on the right-hand side ranges over the pairs of integers  $i<j$ such that  $i$ and $j$ are odd and between  $1$ and $2t-1$. In all the next section, all products denoted by $\displaystyle \prod_{i<j}$ should be interpreted in this way.

\subsection{Last steps of the proof of Theorem~\ref{theoremeintro}}\label{sectiontest}
In order to complete the proof of Theorem~\ref{theoremeintro}, we want to rewrite the product on the right-hand side of \eqref{eqreecrire}. 
 
The strategy to perform this rewriting consist into making a recurrence on the number of principal hooks of the self-conjugate partition involved in \eqref{eqreecrire}.
Let $\lambda$ be a non-empty self-conjugate $2t$-core, and set $\Delta$ its set of principal hook lengths. We denoted by   $h_{1,1}$ the length of its largest principal hook. By deleting this hook in the Ferrers diagram of $\lambda$, we obtain again a (eventually empty) self-conjugate $2t$-core, which we denote by $\lambda'$. We also denote by $\Delta'$ the set of principal hook lengths of  $\lambda'$, and we consider the integers  $\Delta_i'$ which are associated with $\lambda'$ through Definition~\ref{defdeltaiauto2}.

\begin{figure}[h!]\begin{center}
\includegraphics[scale=1.2]{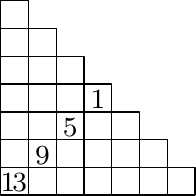} \hspace*{1cm}$\rightarrow$ \hspace*{1cm}\includegraphics[scale=1.2]{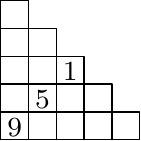}
\caption{\label{fig12n}On the left, a self-conjugate $6$-core  $\lambda$. On the right, the corresponding $6$-core $\lambda'$.}
\end{center}
\end{figure}

In the example of Figure~\ref{fig12n}, we have $\Delta_1 = 5$, $\Delta_3=9$, and $\Delta_5=13$. Moreover, $\Delta_1' = 5$, $\Delta_3'=9$, and $\Delta_5'=1$
 
We now prove the key lemma for our rewriting.
 
\begin{lemma}\label{lemmei_0}
 If $i_0$ is the unique integer such that $\Delta_{i_0}= h_{11}$, then we have:
\begin{multline}\label{eqlemmei_0}
\displaystyle \prod_{i<j}\frac{\displaystyle (2t+\Delta_i) ^2- (2t+\Delta_j)^2 }{\displaystyle (2t+\Delta_i') ^2- (2t+\Delta_j')^2 } \\=\frac{h_{1,1}-2t}{h_{1,1}}\frac{h_{1,1}}{h_{1,1}-2t}\prod_{j \neq i_0} \frac{h_{1,1}-\Delta_j}{h_{1,1}-\Delta_j-4t}\frac{h_{1,1}+\Delta_j+4t}{h_{1,1}+\Delta_j}.
\end{multline}

\end{lemma} 
 
 \begin{proof}
First, notice that 
 \begin{equation*}
 \frac{h_{11}-2t}{h_{11}}\frac{h_{11}}{h_{11}-2t}=1.
 \end{equation*}

To prove \eqref{eqlemmei_0}, we examine the consequences on the numbers $\Delta_i$ of the deletion of the principal hook of length  $h_{1,1}$. As we have  $\Delta' = \Delta \setminus \{h_{1,1}\}$,  for all $j \neq i_0$, we have $\Delta_j= \Delta_j'$. Indeed, the only maximum of the principal hook length congruency classes modulo $2t+2$ which can be changed by this deletion is $\Delta_{i_0}$. We can deduce that:
 \begin{equation*}
 \displaystyle \prod_{i<j}\frac{\displaystyle (2t+\Delta_i) ^2- (2t+\Delta_j)^2 }{\displaystyle (2t+\Delta_i') ^2- (2t+\Delta_j')^2 } =\displaystyle \prod_{j \neq i_0}\frac{(2t+\Delta_{i_0})^2-(2t+\Delta_{j})^2}{(2t+\Delta_{i_0}')^2-(2t+\Delta_{j})^2}.
 \end{equation*}
 Now we have to consider three different cases depending on the value of $h_{1,1}= \Delta_{i_0}$. 
\begin{itemize}
\item If $h_{1,1}>4t$, then according to Remark~\ref{remcle} (ii), $h_{1,1}-4t$ is also a principal hook length of $\lambda$, and so $\Delta_{i_0}' = h_{1,1}-4t$ according to Remark~\ref{remcle} (i). Thus, we have :
\begin{multline*}
\displaystyle \prod_{j \neq i_0}\frac{(2t+\Delta_{i_0})^2-(2t+\Delta_{j})^2}{(2t+\Delta_{i_0}')^2-(2t+\Delta_{j})^2}= \prod_{j \neq i_0}\frac{(2t+h_{1,1})^2-(2t+\Delta_{j})^2}{(h_{1,1}-2t)^2-(2t+\Delta_{j})^2}  \\=\prod_{j \neq i_0} \frac{h_{1,1}-\Delta_j}{h_{1,1}-\Delta_j-4t}\frac{h_{1,1}+\Delta_j+4t}{h_{1,1}+\Delta_j}.
\end{multline*}
\item If $2t<h_{1,1}<4t$, by definition of $\Delta_{i_0}$, we have $\Delta_{i_0}=h_{1,1}=2t+i_0$. According to Remark~\ref{remcle} (i), $2t-i_0$ is not a principal hook length of $\lambda$, so $\Delta_{i_0}' = i_0-2t= h_{1,1}-4t$. The same calculus as in the previous case leads to \eqref{eqlemmei_0}.

 \item If $0<h_{1,1}<2t$, by definition of $\Delta_{i_0}$, we have $\Delta_{i_0}=h_{1,1}=2t-i_0$. So $\lambda'$ contains neither a principal hook length equal to  $2t-i_0 \mod 4t$ nor a principal hook length equal to $2t+i_0 \mod 4t$. Then we have $\Delta_{i_0}' = i_0-2t= -h_{1,1}$ and the end of the proof is the same as formerly.

\end{itemize}
 
 \end{proof}
In the rest of the proof, we need the following notion defined in \cite{HAN}.
 \begin{definition}A finite set of integers $A$ is a $2t$-\emph{compact set} if and only if the following conditions hold:
\begin{itemize}
\item[(i)] $-1,-2, \ldots, -2t+1$ belong to $A$;
\item[(ii)] for all $a \in A$ such that $a \neq -1,-2, \ldots, -2t+1$, we have $a \geq 1$ and $a \not\equiv 0~\mbox{mod~} 2t$;
\item[(iii)] let $b>a \geq 1$ be two integers such that $a \equiv b$ mod $2t$. If $b \in A$, then $a \in A$.
\end{itemize}
\end{definition}
Let A be a $2t$-compact set. An element $a\in A$ is $2t$\emph{-maximal} if for any integer $b>a$ such that $a\equiv b $ mod $2t$, $b \notin A$ ($i.e.$ $a$ is maximal in its congruency class modulo $2t$). The set of $2t$-maximal elements is denoted by $max_{2t}(A)$. It is clear by definition of compact sets that $A$ is uniquely determined by $max_{2t}(A)$. We here recall the following lemma \cite[Lemma 3.5]{PET2}.
\begin{lemma}\label{lemme han} For any $2t$-compact set A, we have:
\begin{equation}\label{eq17}
 \prod_{a \in max_{2t}(A)} \frac{a+2t}{a}=-\prod_{a\in A, a>0}\left(1-\left(\frac{2t}{a} \right)^2\right) .
\end{equation}
\end{lemma}
 
 \begin{lemma}\label{lemmemaxauto} With the same notations as above, we define the set $E $ as:
 \begin{equation}\label{eqlemmemaxauto}
 E:=  \left(\bigcup_{j \neq i_0}\left\{\frac{h_{1,1}-\Delta_j}{2}-2t, \frac{h_{1,1}+\Delta_j}{2}\right\}\right) \cup \{h_{1,1}-2t\}.
 \end{equation}
 Then there exists an unique $2t$-compact set $H$ such that $E=max_{2t}(H)$. Moreover, the subset $H_{>0}$ of positive elements of $H$ is independent of $t$ and made of the hook length of the first column of the Ferrers diagram of $\lambda$,  without $h_{1,1}$.
 \end{lemma}
\begin{proof}
First, notice that, as all the elements of  $\Delta$ are odd, all the elements of $E$ are integers. To prove that the set $E$ is the $max_{2t}(H)$ of a set $H$, it is necessary and sufficient to check that $E$ contains exactly $2t-1$ elements, that all these elements are distinct modulo $2t$ and that none of them is equal to $0$ modulo $2t$; and those three properties come directly from the Definition~\ref{defdeltaiauto2} of the integers  $\Delta_i$. 

Denote by $H_{>0}$  the subset of positive elements of $H$, and $H'$ the set of hook length of the first column of the Ferrers diagram of $\lambda$, with the exception of  $h_{1,1}$.  Recall that we denote by  $\Delta$ the set of principal hook length of  $\lambda$ and by $\Delta^g$ the set of generalized principal hook length of $\lambda$. It remains to show that $H_{>0}= H'$.

According to \cite[Lemme 2.1]{HAN}, we know that if we add $\{-1, -2, \ldots, -t+1\}$ to the set of hook lengths of the first column of a  $t$-core, the result is a $t$-compact set. Consequently, $H'':=H' \cup \{-1, -2, \ldots, -2t+1\}$ is a $2t$-compact set. To show the  required  equality, it is sufficient to show that $H$ and $H''$ have the same maximal elements. To do this, we prove that the elements $x$ of $E$ (which are the maximal elements of $H$) are the maximal elements of  $H''$.

For the element $x := h_{1,1}-2t $ of $E$, three cases can occurs according to the value of $h_{1,1}$. 

If $h_{1,1}>4t$, then according to Remark~\ref{remcle}, $h_{1,1} -4t$ is also a principal hook length, and  $\displaystyle\frac{h_{1,1}+h_{1,1}-4t}{2}= h_{1,1}-2t$ belongs to $H'$, as it is the hook length of the box in the first column of $\lambda$ and in the same row as the box with principal hook length $h_{1,1}-4t$.  Moreover, this element is maximal in  its congruency class  modulo $2t$ in $H'$, as $h_{1,1}$ does not belong to $H'$.

If $4t>h_{1,1}>2t$, according to Remark~\ref{remcle}, $4t-h_{1,1}$ does not belong to $\Delta$. As $4t-h_{1,1}$ is smaller than $h_{1,1}$, according to Lemma~\ref{lemme+-}, $h_{1,1}-4t$  belongs to $\Delta^g$, and so $\displaystyle\frac{h_{1,1}+h_{1,1}-4t}{2}= h_{1,1}-2t$ belongs to $H'$ (by the same argument as in the previous case) and is maximal in its congruency class.
 
If $2t>h_{1,1}>0$, in this case $h_{1,1}-2t$ is strictly between $0$ and $-2t$, and so belongs to $H''$ by definition.

\smallskip

For the elements $\displaystyle \frac{h_{1,1}+\Delta_j}2$ and $\displaystyle \frac{h_{1,1}-\Delta_j-4t}2$, two cases can occur according to the sign of $\Delta_j$.

If $\Delta_j >0$, then we have $\Delta_j <h_{1,1}$. So, $\Delta_j$ belongs to $\Delta$, and $\displaystyle \frac{h_{1,1}+\Delta_j}{2}$ is an element of $H'$. This element is maximal in its congruency class modulo $2t$ in  $H'$. Indeed if we assume that $\displaystyle \frac{h_{1,1}+\Delta_j}{2} +2t$ is an element of $H'$, then we can deduce that $\Delta_j +4t$ belongs to $\Delta$, which contradict the maximality of $\Delta_j$.
 
For $ \displaystyle \frac{h_{1,1}-\Delta_j-4t}2$, we have to consider two sub-cases.
  
Either $\Delta_j+4t > h_{1,1}$, and in this sub-case, $-(\Delta_j +4t)$ belongs to $\Delta^g$ according to the maximality of $\Delta_j$ and Lemma~\ref{lemme+-}, and so $\displaystyle \frac{h_{1,1}-\Delta_j-4t}2$ belongs to $H'$.
  
Or $\Delta_j+4t>h_{1,1}$, in this sub-case, we have the inequality $-2t<\displaystyle \frac{h_{1,1}-\Delta_j-4t}2<0$ and so $\displaystyle \frac{h_{1,1}-\Delta_j-4t}2$ belongs to $H''$. 
   
The fact that all these elements are maximal in their congruency class in $H''$ can be proved in the same way as formerly.
  
If $\Delta_j<0$, the same arguments (taking care of the case where $\Delta_j$ does not belong to $\Delta^g$, as $h_{1,1}$ is smaller than $\Delta_j$ in absolute value) allow us to show that  $\displaystyle \frac{h_{1,1}+\Delta_j}2$ and $\displaystyle \frac{h_{1,1}-\Delta_j-4t}2$ both belong to $H''$. Their maximality can be showed by one of the two previous manners used above.

\end{proof}
 
 \begin{lemma}\label{reecccheck}
Let $\lambda$ be a self-conjugate $2t$-core and $(n_1, \ldots, n_t) \in \mathbb{Z}^t$ such that $\lambda= \phi_1({\bf n})$. The following equality holds:
 \begin{equation}\label{eq48n}
 \prod_{i<j} \left( (4tn_i+2i-1)^2-(4tn_j+2j-1)^2 \right) = \frac{\delta_ \lambda}{c_2}\prod_{h \in \mathcal{H}(\lambda)} \left(1- \frac{2t}{h \, \varepsilon_h}\right),
 \end{equation}
where we recall that the product on the left-hand side ranges over the pairs of integers $i<j$ such that $i$ and $j$ are odd and between $1$ and $2t-1$.
 \end{lemma}
 
 \begin{proof}
This is at this step that we make an induction on the number of principal hooks of the partition $\lambda$, by deleting in the partition at each step of the induction the principal hook of greatest length (that is the hook coming from the box $(1,1)$ of the partition). Here, we denote by $P$ the left-hand side of \eqref{eq48n}, and we use the bijection $\phi_1$ to transform $P$ on a product involving the integers $\Delta_i$, in the same way as when we establish \eqref{eqreecrire}
 
\begin{equation*}
P = \prod_{i<j} \left((2t+\Delta_i) ^2- (2t+\Delta_j)^2 \right).
\end{equation*}

Using the same notations as the ones before Lemma~\ref{lemmei_0}, we define
\begin{equation*}
P' := \prod_{i<j} \left((2t+\Delta_i') ^2- (2t+\Delta_j')^2 \right).
\end{equation*}
Lemma~\ref{lemmei_0} ensures us that:
 \begin{equation*}
 P = \frac{h_{1,1}-2t}{h_{1,1}}\frac{h_{1,1}}{h_{1,1}-2t}\prod_{j \neq i_0} \left(\frac{h_{1,1}-\Delta_j}{h_{1,1}-\Delta_j-4t}\frac{h_{1,1}+\Delta_j+4t}{h_{1,1}+\Delta_j} \right)\times P'.
 \end{equation*}
 
Now, Lemma~\ref{lemmemaxauto} show that the set $E$ defined in \eqref{eqlemmemaxauto} is the $max_{2t}(H)$ of the set $H$, and the positive elements of $H$ are the hook lengths of the first column of the Ferrers diagram of $\lambda$, with the exception of $h_{1,1}$. We can apply Lemma~\ref{lemme han} to show that 
 \begin{multline*}
  \frac{h_{1,1}-2t}{h_{1,1}}\frac{h_{1,1}}{h_{1,1}-2t}\prod_{j \neq i_0} \frac{h_{1,1}-\Delta_j}{h_{1,1}-\Delta_j-4t}\frac{h_{1,1}+\Delta_j+4t}{h_{1,1}+\Delta_j}\\ =   -\frac{h_{1,1}-2t}{h_{1,1}} \prod_{h}\left(1- \frac{2t}{h^2}\right),\hspace*{2cm}
 \end{multline*}
where the product on the right-hand side ranges over the hook lengths of the first column of $\lambda$, without $h_{1,1}$. This right-hand side can be immediately re-write as:
\begin{equation*}
-\prod_{h} \left(1 - \frac{2t}{h \, \varepsilon_h}\right),
\end{equation*}
where the product is over the hook length of all boxes belonging to the principal hook coming from the box  $(1,1)$.

So, we showed that
\begin{equation*}
P= -\prod_{h} \left(1 - \frac{2t}{h \, \varepsilon_h}\right) \times P',
\end{equation*}
with the same condition as above for the product.

The expected result then follows by making the announced recurrence, after remarking that there is exactly   $D(\lambda)$ steps in the recurrence, that each step gives rise to a minus sign, which explains the term $\delta_\lambda$. The base case of the recurrence corresponds to an empty partition $\lambda$ . In this case, we have by definition $\Delta_i = i -2t$ for any odd integer $i$ such that $1 \leq i \leq 2t-1$, and so
\begin{equation*}
\prod_{i<j} \left((2t+\Delta_i') ^2- (2t+\Delta_j')^2 \right)= \prod_{i<j} (i^2-j^2) = \frac{1}{c_2}.
\end{equation*}
 \end{proof}
 
We can now prove Theorem~\ref{theoremeintro}.
 
\begin{proof}[Proof of Theorem~\ref{theoremeintro}]
 By applying the bijection $\phi_1$ to the vectors in the right-hand side of \eqref{eqCcheck}, and by using Lemma~\ref{reecccheck}, we show that for all integers $t \geq 2$, we have:
 \begin{equation}\label{eq56}
 \left(\prod_{i \geq 1} \frac{(1-x^{2i})^{t+1}}{1-x^i}\right)^ {2t-1}= \sum _{\lambda \in SC_{(2t)}} \delta_\lambda \, x^{|\lambda|} \prod_{h \in \mathcal{H}(\lambda)}\left(1-\frac{2t}{h\, \varepsilon_h}\right).
 \end{equation}
 
If a partition  $\lambda$ is self-conjugate,  then the multi-set of its hook lengths of boxes which are strictly below its principal diagonal is the same as the multi-set of its hook lengths of boxes which are strictly below its principal diagonal. It follows that the product $ \displaystyle \prod_{h \in \mathcal{H}(\lambda)}\left(1-\frac{2t}{h\, \varepsilon_h}\right)$ vanishes if $\lambda$ is a self-conjugate partition which is not a $2t$-core. Equation~\eqref{eq56} still holds for all integers  $t\geq 2$ if the sum ranges over all self-conjugate partitions.

Let $m$ be a non-negative integer. The coefficient $C_m(t)$ of $x^m$ on the left-hand side of~\eqref{eq56} is a polynomial in $t$, as is the coefficient $D_m(t)$ of $x^m$ on the right-hand side. Formula~\eqref{eq56} is true for all integers $t \geq 2$, it therefore holds for any complex number $t$, which ends the proof.
  \end{proof}

\subsection{Some applications}\label{section4.4}
We give here some applications of Theorem~\ref{theoremeintro}. First, by setting   $z=-1$, we obtain a new expression for the Jacobi triple product, as it was asked in \cite{HAN09} :
 \begin{eqnarray*}
  \prod_{i \geq 1}(1-x^i)^3 &=& \sum _{\lambda \in \mathcal{P}}  x^{|\lambda|} \prod_{h \in \mathcal{H}(\lambda)}\left(1+\frac{4}{h^2}\right) \hspace*{33pt} \text{ (type $\widetilde{A}$)}\\
  &=& \sum _{\lambda \in SC} \delta_\lambda \, x^{|\lambda|} \prod_{h \in \mathcal{H}(\lambda)}\left(1+\frac{2}{h\, \varepsilon_h}\right) \hspace*{10pt}\text{ (type $\widetilde{C}\check~$)}
 \end{eqnarray*}
 
Second, as its analogous in type $\widetilde{C}$, Theorem~\ref{theoremeintro} makes a connexion between Macdonald formulas of different types. 
 
\begin{Theorem} \label{equiccheck}
The following families of formulas are all generalized by Theorem~\ref{theoremeintro} :
\begin{itemize}
\item[(i)] Macdonald formula \eqref{eqCcheck} in type ${\widetilde{C}}_t\hspace*{-3pt}\check~$ for $t\geq2$;
\item[(ii)] Macdonald formula in type $\widetilde{BC}_t$ for $t\geq 1$:
\begin{equation}\label{eqbc}
(\eta(x^ {1/2})^{-2}\eta(x)^{2t+3})^t= c_2\sum_{\bf v}(-1)^{\sum(v_i-1)/2} x^{\|v\|^2/(4t+2)}\prod_{i<j}(v_i^2-v_j^2),
\end{equation}
where the sum ranges over $t$-tuples ${\bf v}=(v_1, \ldots, v_t) \in \mathbb{Z}^t$ such that $v_i \equiv 2i-1 \mod 4t+2$;
\item[(iii)] Macdonald formula in type $\widetilde{B}_t$ for $t \geq 3$:
\begin{equation}\label{eqb}(\eta(x^ {1/2})^{2}\eta(x)^{2t-3})^z= c_2\sum_{\bf v} (-1)^{\sum v_i} x^{\|v\|^2/(4t-2)}\prod_{i<j}(v_i^2-v_j^2),
\end{equation}
where the sum is over $t$-tuples ${\bf v}=(v_1, \ldots, v_t) \in \mathbb{Z}^t$ such that $v_i \equiv 2i-1 \mod 4t-2$ et $v_1+\cdots+v_t \equiv t^2 \mod 8t-4$.
\end{itemize}
\end{Theorem}

\begin{proof} We do not give the details here, we just mention how Macdonald identities \eqref{eqbc} and \eqref{eqb} can be obtained from \eqref{eqgenccheck2}. 

To obtain \eqref{eqbc}, we set $2t:=2v-1$ in \eqref{eqgenccheck2}, and so consider self-conjugate partitions which are ($2v-1$)-cores. These partitions are in bijection with the vectors involved in Macdonald formula \eqref{eqbc}, that we obtain after the substitution $x \to x^{1/2}$.

To obtain \eqref{eqb}, we set $-2t:= 2v-1$ in \eqref{eqgenccheck2},  and we consider the self-conjugate partitions which can have $2v-1$ as hook length only in their principal diagonal. These partitions are in bijection with the vectors involved in Macdonald formula  \eqref{eqb} that we obtain after the substitution  $x \to x^{1/2}$. 

\end{proof}

\bibliographystyle{plain}
\bibliography{bibliographie}

\end{document}